\documentclass[10pt]{amsart}
\usepackage{amssymb,amsmath,amsthm,latexsym}
\usepackage{mathrsfs}
\usepackage{a4wide}
\usepackage{eucal}
\usepackage[usenames,dvipsnames]{color}
\usepackage[utf8]{inputenc}
\usepackage[T1]{fontenc}
\usepackage{dsfont}
\usepackage{enumerate}

\newtheorem{theorem}{Theorem}[section]
\newtheorem{lemma}[theorem]{Lemma}

\newtheorem{proposition}[theorem]{Proposition}
\newtheorem{problem}[theorem]{Problem}

\newtheorem{corollary}[theorem]{Corollary}

\theoremstyle{definition}

\newtheorem{example}[theorem]{Example}

\theoremstyle{remark}
\newtheorem{remark}[theorem]{Remark}

\numberwithin{equation}{section}
\newcommand{\R}{\mathbb{R}}

\DeclareMathOperator{\supp}{supp}

\newcommand{\nothing}[1]{}

\newcounter{smallromansdash}

{\end{list}}

\newcounter{bigromans} 
  {\end{list}}

\begin{document}

\title[Complemented copies of $c_0$ in spaces $C_p(X,E)$]{On complemented copies of the space $c_0$ in spaces $C_p(X,E)$}
\author[C. Bargetz]{Christian Bargetz}
\address{Universität Innsbruck, Department of Mathematics, Innsbruck, Austria.}
\email{christian.bargetz@uibk.ac.at}
\author[J. K\k{a}kol]{Jerzy K\k{a}kol}
\address{Faculty of Mathematics and Computer Science, A. Mickiewicz University, Pozna\'n, Poland, and Institute of Mathematics, Czech Academy of Sciences, Prague, Czech Republic.}
\email{kakol@amu.edu.pl}
\author[D.\ Sobota]{Damian Sobota}
\address{Kurt G\"odel Research Center, Department of Mathematics, Vienna University, Vienna, Austria.}
\email{damian.sobota@univie.ac.at}

\begin{abstract}
  We study the question for which Tychonoff spaces $X$ and locally convex spaces $E$ the space $C_p(X,E)$ of continuous $E$-valued functions on $X$  contains a complemented copy of the space $(c_0)_p=\{x\in\mathbb{R}^\omega\colon x(n)\to0\}$, both endowed with the pointwise topology. We provide a positive answer for a vast class of spaces, extending classical theorems of Cembranos, Freniche, and Doma\'nski and Drewnowski, proved for the case of Banach and Fr\'echet spaces $C_k(X,E)$. Also, for given infinite Tychonoff spaces $X$ and $Y$, we show that $C_p(X,C_p(Y))$ contains a complemented copy of $(c_0)_p$ if and only if any of the spaces $C_p(X)$ and $C_p(Y)$ contains such a subspace.
\end{abstract}

% 46E40 - spaces of vector-valued functions
% 46E10 - spaces of continuous functions
% 46A03 - locally convex spaces
% 46A08 - barrelled spaces
%\subjclass[2010]{Primary: ... Secondary: ...}
%\keywords{Josefson--Nissenzweig theorem, complementability of $c_0$, locally convex spaces, vector-valued functions, separately continuous functions.}

%\thanks{The first and the third names authors are supported by the Austrian Science Fund (FWF):      I 4570-N. The second named author is supported by the GA\v{C}R project 20-22230L  RVO: 67985840. }
\maketitle

\section{Introduction}
For   a topological space $X$ and a locally convex space $E$, by  $C_p(X,E)$ and $C_k(X,E)$ we denote the space   of all continuous $E$-valued functions on $X$ endowed with the pointwise topology (i.e. the topology inherited from the product $E^X$) and the compact-open topology, respectively. Similarly, if $X$ is a compact space and $E$ is a Banach space, by $C(X,E)$ we mean the Banach space of all continuous $E$-valued functions on $X$ endowed with the sup-norm topology.
If $E=\mathbb{R}$, then we simply write $C(X)$, $C_p(X)$ and $C_k(X)$ instead of $C(X,\mathbb{R})$, $C_p(X,\mathbb{R})$ and $C_k(X,\mathbb{R})$, respectively.

Recall the following result due to Cembranos (\cite{Cem}) and Freniche (\cite[Corollary 2.5]{Fre84}).

\begin{theorem}[Cembranos, Freniche]\label{Cembranos}
  For an infinite compact space $X$ and  an infinite-dimensional Banach space $E$, the Banach space $C(X,E)$  contains a complemented copy of the Banach space $c_{0}$.
\end{theorem}

Theorem \ref{Cembranos} was extended by Doma\'nski and Drewnowski (\cite[Theorem~1, p.~260]{DD}) to Fr\'echet spaces, i.e. complete and metrisable locally convex spaces, of the form $C_{k}(X, E)$.

\begin{theorem}[Doma\'nski--Drewnowski]\label{dom-drew}
Let $X$ be a Tychonoff space which contains an infinite compact subset and $E$ a Fr\'echet space which is not Montel. Then the Fr\'echet space $C_{k}(X,E)$ contains a complemented copy of the space $c_{0}$.
\end{theorem}

Both of the theorems have important consequences. E.g., since for each Tychonoff space $X$ and locally compact space $Y$, the spaces $C_{k}(X\times Y)$ and $C_{k}(X,C_{k}(Y))$ are linearly homeomorphic (in short, \textit{isomorphic}), see \cite[Corollary 2.5.7]{McCoy}, Theorem \ref{Cembranos} yields that for infinite compact spaces $X$ and $Y$ the Banach space $C(X\times Y)$ contains a complemented copy of $c _{0}$. Similarly, if $X$ is a Tychonoff space containing an infinite compact subset and $Y$ is a locally compact and $\sigma$-compact Tychonoff space, then $C_{k}(Y)$ is an infinite-dimensional Fr\'echet space which is not a Montel space (\cite[Theorem 11.8.7]{Jarchow}), and hence Theorem \ref{dom-drew} implies that $C_{k}(X\times Y)\approx C_{k}(X,C_{k}(Y))$ contains a complemented copy of $c_{0}$.

Let the space $(c_0)_p=\{x\in\mathbb{R}^\omega\colon\ x(n)\to 0\}$ be endowed with the pointwise topology inherited from $\mathbb{R}^\omega$. The present paper deals with the following question, motivated by Theorems \ref{Cembranos} and \ref{dom-drew}: ($*$) \emph{When does a given space $C_{p}(X,E)$ contain a complemented copy of $(c_{0})_p$?}

We first answer question ($*$) in the case of spaces of the form $C_p(X,C_p(Y))$---in Theorem \ref{thm:cpxcpy_cps_compl} we show that $C_p(X,C_p(Y))$ contains a complemented copy of $(c_0)_p$ if and only if any of the spaces $C_p(X)$ and $C_p(Y)$ contains such a copy. For characterisations and examples of classes of Tychonoff spaces $X$ for which $C_p(X)$ contains a complemented copy of $(c_0)_p$, see e.g. \cite{BKS}, \cite{KMSZ}, \cite{KSZproc}, \cite{KSZfm}. 

Regarding a more general case, we show in Theorem \ref{the1} that for every Tychonoff space $X$ containing an infinite compact subset and barrelled locally convex space $E$ with the Josefson--Nissenzweig property (see Section \ref{sec:jnp} for the definition) the space $C_{p}(X, E)$ contains a complemented copy of $(c_{0})_{p}$. This not only generalises Theorem \ref{Cembranos} (see Corollary \ref{first}) and provides a $C_p$-variant of Theorem \ref{dom-drew} (Corollary \ref{variant_dd}), but also has several applications. For example, we get that if $X$ is an infinite Tychonoff space and $Y$ is an infinite compact space, then the space $C_{p}(X,C(Y))$ contains 
a complemented copy of $\mathbb{R}^{\omega}$ or a complemented copy of $(c_{0})_{p}$ (Corollary \ref{mon}; see also Corollary \ref{second_cor} for a $C_k(X,C_k(Y))$-version). We immediately get from this that the space  $C_{p}(X,C(Y))$ is isomorphic to  $C_{p}(X,C(Y))\times\mathbb{R}$ (Corollary \ref{second}), which extends the number of known cases of locally convex spaces $E$ for which $E\approx E\times\mathbb{R}$, see the remark following Corollary \ref{second} and concerning a problem of Arkhangel'ski.

For  the case of spaces of the form $C_{p}(X,E_{w})$, where $E_{w}$ means a locally convex spaces $E$ with its weak topology, we show in Theorem \ref{the1-1} that, for every Tychonoff space $X$ containing an infinite compact subset and locally convex space $E$ containing a copy of the Banach space $\ell_1$, the space $C_p(X,E_w)$ contains a complemented copy of $(c_0)_p$. Hence, we get that if  $X$ is  a Tychonoff space which contains an infinite compact subset and  $Y$ is  an infinite Tychonoff space containing a non-scattered compact subset, then also  $C_{p}(X,C_{k}(Y)_{w})$ contains a complemented copy of $(c_{0})_{p}$, see  Corollary \ref{sca}. This applies to show that, e.g., there  is no continuous linear surjection  from  $C_{p}(\beta\omega, C_{p}(\beta\omega))$ onto $C_{p}(\beta\omega, C(\beta\omega)_{w})$, as well as that for every infinite compact space $X$ the spaces $C_p(X)$ and $C(X)_w$ are not isomorphic (which is a special case of \cite[Corollary 3.2]{KM}, cf. Remark \ref{rem:km_cp}).

\section{Preliminaries}

The cardinality of a set $X$ is denoted by $|X|$. By $\omega$ we denote the cardinality of the space of natural numbers $\mathbb{N}$ and, as usual, we identify $\omega$ with $\mathbb{N}$. $\beta\omega$ denotes the \v{C}ech--Stone compactification of $\omega$, and $\omega^*=\beta\omega\setminus\omega$.

We assume that all topological spaces we consider are \textbf{Tychonoff}, that is, completely regular and Hausdorff. If $(X,\tau)$ is a topological space (with the topology $\tau$) and $Y\subseteq X$, then by $\tau\restriction Y$ we mean the restriction of $\tau$ on $Y$, that is, $\tau\restriction Y=\{U\cap Y\colon U\in\tau\}$.

For a locally convex space $E$,  $E'$ denotes its topological dual and $E_w$ denotes the pair $(E,\sigma(E,E'))$), i.e. the space $E$ endowed with its weak topology. 

If $X$ is a Tychonoff space, then $C_p(X)=(C_p(X))_w$ holds. We write $L_p(X)=C_p(X)'$.  Recall that each $\varphi\in L_p(X)$ may be uniquely represented as a finite linear combination of point-measures on $X$, i.e. $\varphi=\sum_{x\in X}\alpha_x^\varphi\delta_x$, where $\alpha_x^\varphi\in\mathbb{R}$ and  $\alpha_x^\varphi\neq 0$ only for finitely many $x\in X$. If $A\subseteq X$, then we set $\varphi\restriction A=\sum_{x\in A}\alpha_x^\varphi\delta_x$. We put $\supp(\varphi)=\{x\in X\colon\ \alpha_x^\varphi\neq0\}$, and define \textit{the norm} $\|\varphi\|$ of $\varphi$ by the formula $\|\varphi\|=\sum_{x\in\supp(\varphi)}|\alpha_x^\varphi|$.

Note that if $\mathcal{P}$ is a family of semi-norms on a locally convex space $E$ inducing the topology of $E$, then the pointwise topology of $C_p(X,E)$ is induced by the family $\mathcal{Q}$ of semi-norms given for every $f\in C_p(X,E)$ by the formula: $q_{x_1,\ldots,x_n;p}(f)=\max_{i=1,\ldots,n}p(f(x_i))$, where $p\in\mathcal{P}$ and $x_1,\ldots,x_k\in X$; see \cite[Section I.2]{Schmets}. In particular, if $E=C_p(Y)$ for Tychonoff $Y$, then the topology on $E$ is induced by the semi-norms: $p_{y_1,\ldots,y_m}(f)=\max_{i=1,\ldots,m}|f(y_i)|$, with $y_1,\ldots,y_m\in Y$, whence the topology on $C_p(X,E)$ is given by the semi-norms of the form:
\[q_{x_1,\ldots,x_n;y_1,\ldots,y_m}(f)=\max_{i=1,\ldots,n}\max_{j=1\ldots,m}|f(x_i)(y_j)|,\]
where $x_1,\ldots,x_n\in X$ and $y_1,\ldots,y_m\in Y$.

If $E$ and $F$ are both locally convex,  we write $E\approx F$ if $E$ and $F$ are \textit{isomorphic} (i.e. linearly homeomorphic). We also say that $E$ \textit{contains a complemented copy of} $F$ (or that $F$ \textit{is complemented in} $E$) if there are closed linear subspaces $E_1$ and $E_2$ of $E$ such that $E$ is a direct algebraic sum of $E_1$ and $E_2$ (i.e. $E=E_1\oplus E_2$ as a vector space and $E_1\cap E_2 =\{0\}$), $E_1$
is isomorphic to $F$, and the natural projection from $E$ onto $E_1$ is continuous.

By $S$ we denote the non-trivial convergent sequence, i.e. the space $S=\{0\}\cup\{1/n\colon n\in\omega, n>0\}$ with the usual topology inherited from the real line. By $c_0$ we mean the usual Banach space of all real-valued sequences convergent to $0$, i.e. $c_0=\{x\in\mathbb{R}^\omega\colon x(n)\to0\}$ endowed with the supremum norm. The symbol $(c_0)_p$ denotes the space $c_0$  with the pointwise topology inherited from the product $\mathbb{R}^\omega$. Note that $(c_0)_p\approx C_p(S)$ (see \cite[p.397]{vM2}).

\section{$C_p(X\times Y)$ versus $C_p(X,C_p(Y))$\label{sec:cpxy_vs_cpxcpy}}

In this section we will briefly recall standard and folklore facts concerning spaces of separately continuous functions on products $X\times Y$ of two Tychonoff spaces and their relations to the spaces $C_p(X,C_p(Y))$, which will be useful in the next sections.

Fix Tychonoff spaces $X$ and $Y$. Recall that a function $f\colon X\times Y\to\mathbb{R}$ is \textit{separately continuous} if the functions $Y\ni y\mapsto f(x_0,y)$ and $X\ni x\mapsto f(x,y_0)$ are continuous for every $x_0\in X$ and $y_0\in Y$. Of course, every continuous function is separately continuous, but the converse may not hold (e.g. for every infinite compact spaces $X$ and $Y$ there is a separately continuous non-continuous $f\colon X\times Y\to\R$). By $SC(X\times Y)$ we denote the space of all real-valued separately continuous functions on $X\times Y$ and by $SC_p(X\times Y)$ we mean $SC(X\times Y)$ endowed with the pointwise topology. Recall that $C(X\times Y)$ is a linear subspace of $SC(X\times Y)$ and that $C_p(X\times Y)$ is dense in $SC_p(X\times Y)$.

Let the mapping $A\colon SC_p(X\times Y) \to C_p(X, C_p(Y))$ by defined as follows: for every $f\in SC_p(X\times Y)$, $A(f)$ is the unique mapping in $C_p(X,C_p(Y))$ with $A(f)(x)(y)=f(x,y)$ for every $x\in X$ and $y\in Y$, that is,
\[SC_p(X\times Y)\ni f\substack{A\\\longmapsto\\~}\big[x\mapsto [y\mapsto f(x,y)]\big]\in C_p(X,C_p(Y)).\]
The following theorem is folklore, cf. e.g. \cite[Lemma 1.2]{GS00}.

\begin{theorem}\label{thm:sepcont_iso}
  The mapping $A$ is a well-defined isomorphism between the spaces $SC_p(X\times Y)$ and $C_p(X,C_p(Y))$.
\end{theorem}

The theorem implies that $C_p(X\times Y)$ embeds as a dense linear subspace of the space $C_p(X,C_p(Y))$.

Let $X$ and $Y$ be Tychonoff spaces. By $\sigma$ let us denote the weak topology on $X\times Y$ generated by the family of all separately continuous functions. This topology was investigated by Henriksen and Woods \cite{HW}, who proved among others that:
\begin{itemize}
\item $(X\times Y,\sigma)$ is a Tychonoff space (\cite[Theorem 4.8]{HW});
\item $(X\times Y,\sigma)$ is not pseudocompact provided that both $X$ and $Y$ are infinite (\cite[Corollary 5.9.b]{HW});
\item $C(X\times Y,\sigma)=SC(X\times Y)$, i.e. every continuous function on $(X\times Y,\sigma)$ is separately continuous on $X\times Y$ and \textit{vice versa} (\cite[Proposition 2.1]{HW}).
\end{itemize}

By the last listed result and Theorem \ref{thm:sepcont_iso}, we immediately get the following theorem, which allows us to treat spaces of the form $C_p(X,C_p(Y))$ the same way as spaces of the form $C_p(X\times Y)$.

\begin{theorem}\label{thm:cpxysigma_cpxcpy}
  Let $X$ and $Y$ be Tychonoff spaces. Then, the space $C_p(X\times Y,\sigma)$ is isomorphic to the space $C_p(X,C_p(Y))$.
\end{theorem}

Since for every Tychonoff spaces $X$ and $Y$ the spaces $C_p(X\times Y,\sigma)$ and $C_p(Y\times X,\sigma)$ are isomorphic, Theorem \ref{thm:cpxysigma_cpxcpy} implies trivially a well-known fact that the spaces $C_p(X,C_p(Y))$ and $C_p(Y,C_p(X))$ are isomorphic, too.

The following folklore theorem follows immediately from the main result of Uspensky \cite{Usp82}.

\begin{theorem}\label{thm:cpxy_not_iso_cpxcpy}
  Let $X$ and $Y$ be infinite pseudocompact spaces such that the product $X\times Y$ is pseudocompact. Then, the space $C_p(X\times Y)$ cannot be mapped onto the space $C_p(X,C_p(Y))$ by a continuous linear map. In particular, $C_p(X\times Y)$ is not isomorphic to $C_p(X,C_p(Y))$.
\end{theorem}

Recall that the product of a compact space and a pseudocompact space is pseudocompact.

\begin{corollary}\label{cor:compact_pseudocompact}
  If $X$ is an infinite compact space and $Y$ an infinite pseudocompact space, then the space $C_p(X\times Y)$ cannot be  mapped onto the space $C_p(X,C_p(Y))$ by a continuous linear map.
\end{corollary}

\begin{corollary}\label{cor:compact_cpxy_cpxcpy}
  If $X$ and $Y$ are infinite compact spaces, then $C_p(X\times Y)$ is not isomorphic to $C_p(X,C_p(Y))$.
\end{corollary}

Note that one cannot drop the compactness assumption in Corollary \ref{cor:compact_cpxy_cpxcpy}. Indeed,  if $X$ and $Y$ are infinite and $Y$ is discrete, then the canonical mapping $A\colon C_{p}(X\times Y)\to C_{p}(X,C_p(Y))$ defined above is easily seen to be an isomorphism. 
Also, \cite[Corollary 6.15]{HW} states that if $X$ is a (locally) separable Tychonoff space and $Y$ is a P-space, then $C(X\times Y)=SC(X\times Y)$ and hence $\sigma$ coincides with the product topology on $X\times Y$---it follows by Theorem  \ref{thm:cpxysigma_cpxcpy} that the spaces $C_p(X\times Y)$ and $C_p(X,C_p(Y))$ are isomorphic. In particular, if $X$ is a metric compact space or $X=\beta\omega$ and $Y$ a P-space, then $C_p(X\times Y)\approx C_p(X,C_p(Y))$.

\begin{remark}\label{rem:km_cp}
  Krupski and Marciszewski \cite[Corollary 3.2]{KM} proved that for all infinite compact spaces $X$ and $Y$ the locally convex spaces $C_p(X)$ and $C(Y)_w$ are never isomorphic. It follows that $C_p(X\times Y)$ and $C(X\times Y)_w$ are not isomorphic. Since $C(X\times Y)_w$ is isomorphic to $C(X,C(Y))_w$, we get that $C_p(X\times Y)$ is not isomorphic to $C(X,C(Y))_w$. Thus, Corollary \ref{cor:compact_cpxy_cpxcpy} may be thought of as a $C_p$-counterpart of the result of Krupski and Marciszewski. This observation leads to the following relevant question, some issues of which we will study in Section \ref{sec:c0p_cpxew}.
\end{remark}

\begin{problem}\label{prob:cpxy_cpxcyw}
  For which infinite compact spaces $X$ and $Y$ are the spaces $C_p(X\times Y)$ and $C_p(X,C(Y)_w)$ isomorphic?
\end{problem}

Note that if $X$ and $Y$ are both infinite countable compact spaces, then $C_p(X\times Y)$ is metrisable, whereas $C_p(X,C(Y)_w)$ is not, so these function spaces are not isomorphic.

Let us now provide an example and a question motivated by it.

\begin{example}
  Let $X=[0,1]^\omega$ be endowed with the product topology. From Corollary \ref{cor:compact_cpxy_cpxcpy} it follows that $C_p(X\times X)$ is not isomorphic to $C_p(X,C_p(X))$. However, since $C_p(X\times X)$ is isomorphic to $C_p(X)$, $C_p(X\times X)$ is isomorphic to a complemented subspace of $C_p(X,C_p(X))$ (by~\cite[Theorem I.4.4]{Schmets}). A similar argument may be provided for any uncountable metrisable totally disconnected compact space $X$, e.g. for the Cantor space $2^\omega$.
\end{example}

\begin{problem}\label{problem:cpxx_compl}
  For which infinite compact spaces $X$, is the space $C_p(X\times X)$ isomorphic to a complemented subspace of the space $C_p(X,C_p(X))$? In particular, does there exist an infinite compact space $X$ such that $C_p(X)$ is not isomorphic to $C_p(X\times X)$ and $C_p(X\times X)$ is isomorphic to a complemented subspace of $C_p(X,C_p(X))$?
\end{problem}

In Corollary \ref{cor:cpx_not_jnp_cpxx_not_compl} we describe a class of infinite compact spaces $X$ such that $C_p(X\times X)$ is not complemented in $C_p(X,C_p(X))$. Note that if $X$ is an infinite compact space such that $C_p(X\times X)$ is complemented in $C_p(X,C_p(X))$, then the latter may be linearly mapped onto the former---a phenomenon which by Corollary \ref{cor:compact_pseudocompact} cannot occur in the reverse direction.

We finish the section with the following observation.

\begin{proposition}\label{pro}
  Let $E$ be a normed space and $X$ a Tychonoff space. Then, the original pointwise topology  of $C_{p}(X,E)$  coincides with the  weak topology of $C_{p}(X,E)$  if and only if $E$ is finite-dimensional space.
\end{proposition}
\begin{proof} If $E$ is finite-dimensional, say $E$ is isomorphic to $\mathbb{R}^{n}$ for some $n\in\omega$, then we have: $C_{p}(X,E)\approx C_{p}(X,\mathbb{R}^{n})\approx C_p(X)^n$.
Hence, the pointwise topology and weak topology coincide in $C_p(X,E)$, as they do on $C_p(X)$. Conversely, if both of the topologies coincide, then $E$ is finite-dimensional. Indeed, by \cite[Theorem I.4.4]{Schmets} $E$ is isomorphic to a complemented subspace of $C_{p}(X,E)$ and the weak topology on this complemented subspace agrees with the subspace topology inherited from the weak topology of $C_p(X,E)$. Since $E$ is a normed space on which the norm topology and the weak topology agree, it is finite-dimensional.
\end{proof}

\begin{corollary}
If $X$ is a Tychonoff space and $Y$ is an infinite compact space, then the spaces $C_{p}(X\times Y)$ and $C_{p}(X,C(Y))$  are not isomorphic.
\end{corollary}

\section{The Josefson--Nissenzweig Property\label{sec:jnp}}

Following Banakh and Gabriyelyan~\cite{BG}, we say that a locally convex space $E$ has \textit{the Josefson--Nissenzweig property}, or \textit{the JNP} in short, if the identity map
\[
  (E',\sigma(E',E)) \to (E', \beta^*(E',E))
\]
is not sequentially continuous, i.e. if the dual space $E'$ contains a weak$^*$ null sequence which does not converge in the topology~$\beta^*(E',E)$. The topology $\beta^*(E',E)$ can be described as the topology of uniform convergence on so-called \textit{barrel-bounded sets}, i.e. on sets which have the property that they are absorbed by all barrels---see~\cite[p.~3]{BG} for more details. Note that the JNP is preserved by isomorphisms, since the transpose of an isomorphism is also an isomorphism for the above topologies---see e.g. Corollary 8.6.6 in~\cite[p.~161]{Jarchow}. By the famous Josefson--Nissenzweig Theorem, see e.g.~\cite{Josefson,Nissenzweig,BD}, every infinite-dimensional Banach space has the JNP. On the other hand, no Montel space enjoys the JNP. Indeed: First, since Montel spaces are barrelled, the topologies~$\beta^*(E',E)$ and $\beta(E',E)$, where the strong topology~$\beta(E', E)$ is the topology of uniform convergence on bounded subsets of~$E$, coincide. Second, the strong dual of a Montel space is again Montel and hence on bounded sets the topologies $\sigma(E',E)$ and $\beta(E',E)$ coincide. In~\cite{BLV}, Bonet,  Lindström and Valdivia showed that for Fr\'{e}chet spaces this is already a characterisation of the JNP, i.e. a Fr\'{e}chet space has the JNP if and only if it is not Montel; see also \cite{LS} for a similar characterisation of non-Schwarz Fr\'echet spaces.

In the case of $C_p(X)$-spaces, Banakh and Gabriyelyan proved the following characterisation of the Josefson--Nissenzweig property.

\begin{theorem}[{\cite[Theorem 3.4]{BG}}]\label{thm:bg_cpx_jnp}
For every Tychonoff space $X$, the space $C_p(X)$ has the JNP if and only if there exists a sequence of functionals $(\varphi_n)_{n\in\omega}$ in the dual space $L_p(X)$ of $C_p(X)$ such that $\varphi_n(f)\to0$ for every $f\in C(X)$ and $\|\varphi_n\|=1$ for every $n\in\omega$.
\end{theorem}

If a sequence $(\varphi_n)_{n\in\omega}$ of functionals on a given space $C_p(X)$ has the properties as in the above theorem, then it will be called \textit{a JN-sequence} on $X$. 
The existence of JN-sequences on arbitrary spaces was intensively studied in \cite{BKS}, \cite{KMSZ}, \cite{KSZproc}, and \cite{KSZfm}, where several classes of Tychonoff spaces were recognized to (not) admit such sequences (e.g. it appears that for every infinite compact spaces $X$ and $Y$ their product admits a JN-sequence, see \cite[Theorem 1.2]{KSZproc}). Note that if $X$ is pseudocompact and $C_p(X)$ has the JNP, then every JN-sequence on $X$ satisfies already the conclusion of the Josefson--Nissenzweig theorem for the Banach space $C(X)$ (by virtue of the Riesz representation of continuous functionals on $C(X)$).

The following theorem provides a simple criterion for a space $C_p(X\times Y,\sigma)$ to have the JNP.

\begin{theorem}\label{thm:cpxysigma_jnp}
  Let $X$ and $Y$ be infinite Tychonoff spaces. Then, the space $C_p(X\times Y,\sigma)$ has the JNP if and only if the space $C_p(X)$ has the JNP or the space $C_p(Y)$ has the JNP.
\end{theorem}
\begin{proof}
  Assume first that $C_p(X\times Y,\sigma)$ has the JNP and let $(\varphi_n)_{n\in\omega}$ be a JN-sequence on $(X\times Y,\sigma)$. Put $\mathcal{S}=\bigcup_{n\in\omega}\supp(\varphi_n)$. By \cite[Proposition 4.1]{KMSZ}, every continuous function $f\in C(X\times Y,\sigma)$ is bounded on $\mathcal{S}$, that is, $f[\mathcal{S}]$ is a bounded subset of $\mathbb{R}$. By \cite[Corollary 5.9.a]{HW}, $\mathcal{S}$ must be contained in finitely many vertical sections or finitely many horizontal sections of the product $X\times Y$.

  Let us assume that the former happens, i.e. there are $x_1,\ldots,x_m\in X$ such that $\mathcal{S}\subseteq\bigcup_{i=1}^m\{x_i\}\times Y$. Since $\|\varphi_n\|=1$ for every $n\in\omega$, there are $i_0\in\{1,\ldots,m\}$ and a subsequence $(\varphi_{n_k})_{k\in\omega}$ such that $\|\varphi_{n_k}\restriction\{x_{i_0}\}\times Y\|\ge 1/m$ for every $k\in\omega$. We claim that the sequence $(\psi_k)_{k\in\omega}$ defined for every $k\in\omega$ by the formula
  \[\psi_k=(\varphi_{n_k}\restriction\{x_{i_0}\}\times Y)\big/\|\varphi_{n_k}\restriction\{x_{i_0}\}\times Y\|\]
  is a JN-sequence on $\{x_{i_0}\}\times Y$. Obviously, $\|\psi_k\|=1$ for every $k\in\omega$. Let $f\in C_p(\{x_{i_0}\}\times Y)$. Since $X$ is Tychonoff, there is a function $g\in C_p(X)$ such that $g(x_{i_0})=1$ and $g(x_i)=0$ for every $i\in\{1,\ldots,m\}\setminus\{i_0\}$. For every $(x,y)\in X\times Y$ we put: $F(x,y)=g(x)\cdot f(x_{i_0},y)$. 
  It follows that $F\in SC(X\times Y)=C(X\times Y,\sigma)$, $F(x_{i_0},\cdot)\equiv f(x_{i_0},\cdot)$, and $F(x_i,\cdot)\equiv 0$ for every $i\in\{1,\ldots,m\}\setminus\{i_0\}$. Since $(\varphi_{n_k})_{k\in\omega}$ is also a JN-sequence on $(X\times Y,\sigma)$ and $\mathcal{S}\subseteq\bigcup_{i=1}^m\{x_i\}\times Y$, we get that $|\psi_k(f)|\le |\varphi_{n_k}(F)|\cdot m\to 0$ 
  as $k\to\infty$. This proves that $(\psi_k)_{k\in\omega}$ is a JN-sequence on $\{x_{i_0}\}\times Y$---it follows easily that $C_p(Y)$ has the JNP, too.

  If $\mathcal{S}$ is contained in finitely many horizontal sections of the product $X\times Y$, then we proceed similarly and prove that $C_p(X)$ has the JNP.

  \medskip

  Let us assume now that, conversely, $C_p(X)$ has the JNP, that is, there is a JN-sequence $(\varphi_n)$ on $X$. For each $n\in\omega$ write: $\varphi_n=\sum_{i=1}^{k_n}\alpha_i^n\delta_{x_i^n}$. 
  Fix $y\in Y$ and define the functional $\psi_n$ on $C_p(X\times Y,\sigma)$ by the similar formula: $\psi_n=\sum_{i=1}^{k_n}\alpha_i^n\delta_{(x_i^n,y)}$. 
  We claim that $(\psi_n)_{n\in\omega}$ is a JN-sequence on $(X\times Y,\sigma)$. Of course, $\|\psi_n\|=1$ for every $n\in\omega$. Let $f\in C_p(X\times Y,\sigma)=SC_p(X\times Y)$. Then, the function $F\colon X\to\mathbb{R}$ defined for every $x\in X$ by $F(x)=f(x,y)$ is continuous on $X$ and thus $\lim_{n\to\infty}\varphi_n(F)=0$. Since for every $n\in\omega$ we have $\psi_n(f)=\varphi_n(F)$, we get that $\lim_{n\to\infty}\psi_n(f)=0$, too. It follows that the space $C_p(X\times Y,\sigma)$ has the JNP.

  We prove similarly that if $C_p(Y)$ has the JNP, then $C_p(X\times Y,\sigma)$ has the JNP.
\end{proof}

The first part of the proof suggests the following result.

\begin{proposition}
Let $X$ and $Y$ be infinite Tychonoff spaces. Let $(\varphi_n)_{n\in\omega}$ be a sequence in $L_p(X\times Y,\sigma)$ such that $\|\varphi_n\|=1$ for every $n\in\omega$. Set $\mathcal{S}=\bigcup_{n\in\omega}\supp(\varphi_n)$. If $\mathcal{S}$ is contained in neither finitely many vertical sections nor finitely many horizontal sections, then there exists $f\in SC(X\times Y)$ such that $\limsup_{n\to\infty}|\varphi_n(f)|>0$.
\end{proposition}

We will now study the Josefson--Nissenzweig property of the spaces $C_p(X,E)$ for $X$ Tychonoff and $E$ locally convex. Note that by \cite[Proposition~19]{fekanew} the topological dual space of $C_p(X,E)$ is algebraically isomorphic to the space $L_p(X)\otimes E'$, so each functional $\varphi\in C_p(X,E)'$ may be represented as a finite sum of tensor products 
$\sum_{i,j} \alpha_{i,j} (\delta_{x_i} \otimes e_j')$, 
where $x_i\in X$ and $e_j'\in E'$. Hence, if $E=C_p(Y)$ for some Tychonoff space $Y$, then each $\varphi\in C_p(X,E)'$ may be written as a finite sum $\sum_{i,j}\alpha_{i,j}(\delta_{x_i}\otimes\delta_{y_j})$. 
Let us define \textit{the norm} of $\varphi$ (in $L_p(X)\otimes L_p(Y)$) in the following natural way: $\|\varphi\|=\sum_{i,j}|\alpha_{i,j}|$.

Let $X$ and $Y$ be Tychonoff spaces and assume that $C_p(X\times Y,\sigma)$ has the JNP, that is, by Theorem \ref{thm:bg_cpx_jnp}, there is a sequence $(\varphi_n)_{n\in\omega}$ in $L_p(X\times Y,\sigma)$ such that $\varphi_n(f)\to0$ for every $f\in C_p(X\times Y,\sigma)$ and $\|\varphi_n\|=1$ for every $n\in\omega.$ For each $n\in\omega$ write: $\varphi_n=\sum_{i=1}^{k_n}\alpha_i^n\delta_{(x_i^n,y_i^n)}$. 
By Theorem \ref{thm:sepcont_iso}, the operator $A\colon C_p(X\times Y,\sigma)\to C_p(X,C_p(Y))$ is an isomorphism, so the sequence $(\psi_n)_{n\in\omega}$ given for every $n\in\omega$ by the formula:
\[
  \psi_n=A'(\varphi_i) = \sum_{i=1}^{k_n}\alpha_i^n((\delta_{x_i^n}\otimes\delta_{y_i^n})\circ A),
\]
where $A'$ is the transpose of $A$, is weakly$^*$ convergent to $0$ (see \cite[Theorem 21.5]{KN}). Obviously, $\|\psi_n\|=1$ for every $n\in\omega$.

Let us now assume that, conversely, the space $C_p(X,C_p(Y))$ admits a weak$^*$ null sequence $(\psi_n)_{n\in\omega}$ in $L_p(X)\otimes L_p(Y)$ such that $\|\psi_n\|=1$ for every $n\in\omega$. By a very similar argument as above (this time applied to the transpose of $A^{-1}$), one may obtain a JN-sequence on the space $(X\times Y,\sigma)$. This way we obtain the following

\begin{proposition}\label{prop:jnp_cpxcpy}
For every Tychonoff spaces $X$ and $Y$ the space $C_p(X\times Y,\sigma)$ has the JNP if and only if there exists a sequence $(\psi_n)_{n\in\omega}$ in the dual $L_p(X)\otimes L_p(Y)$ which is weakly$^*$ convergent to $0$ and such that $\|\psi_n\|=1$ for every $n\in\omega$.
\end{proposition}

Since the Josefson--Nissenzweig property for locally convex spaces is preserved under topological isomorphisms, Proposition \ref{prop:jnp_cpxcpy} and Theorem \ref{thm:bg_cpx_jnp} yield the following characterisation of the property for $C_p(X,C_p(Y))$-spaces.

\begin{proposition}\label{prop:cpxcpy_jnp_seq}
  For every Tychonoff space $X$ and $Y$ the space $C_p(X,C_p(Y))$ has the JNP if and only if there exists a sequence $(\psi_n)_{n\in\omega}$ in $L_p(X)\otimes L_p(Y)$ such that $\|\psi_n\|=1$ for every $n\in\omega$ and $\psi_n(f)\to0$ for every $f\in C_p(X,C_p(Y))$.
\end{proposition}

The following corollary is an immediate consequence of Theorems \ref{thm:cpxysigma_jnp} and \ref{thm:cpxysigma_cpxcpy}.

\begin{corollary}
  For every infinite Tychonoff spaces $X$ and $Y$, the space $C_p(X,C_p(Y))$ has the JNP if and only if $C_p(X)$ has the JNP or $C_p(Y)$ has the JNP.
\end{corollary}

For general locally convex spaces $E$, we still have one of the implications observed in the above corollary.

\begin{proposition}\label{JNPCpXE}
  Let $X$ be a Tychonoff space and let $E$ be a locally convex space. If   $C_p(X)$ has the JNP or $E$ has the JNP, then the same is true for the space $C_p(X,E)$.
\end{proposition}

\begin{proof}
  Assume that $E$ has the Josefson--Nissenzweig property. This allows us to pick a weak$^*$ null sequence $(\varphi_n)_{n\in\omega}$ in $E'$ and a barrel-bounded set $B\subset E$ on which it does not converge to zero uniformly. Let $f\in C_p(X)$ and $x\in X$ be such that $f(x)=1$. For each $n\in\omega$ set $\psi_n = \delta_x \otimes \varphi_n$; of course, $\psi_n\in L_p(X)\otimes E'$. Note that for every $g\in C_p(X,E)$ we have $\langle \psi_n, g\rangle = \langle \varphi_n, g(x) \rangle \to 0$ 
  as $n\to\infty$, so $(\psi_n)_{n\in\omega}$ is a weak$^*$ null sequence in $(C_p(X,E))'$.

We now exhibit a barrel-bounded subset of $C_p(X,E)$ on which $(\psi_n)_{n\in\omega}$ does not converge uniformly. The tensor product $C_p(X)\otimes E$ embedds (as a vector space) into the space $C_p(X,E)$ by the formula
\[C_p(X)\otimes E\ni\sum_if_i\otimes e_i\longmapsto\big[y\mapsto\sum_if_i(y)e_i\big]\in C_p(X,E).\]
Thus, the subspace of $C_p(X)\otimes E$ given by the formula $f\otimes E=\{f\otimes e\colon\ e\in E\}$ 
may be thought of as a subspace of $C_p(X,E)$. Moreover, by \cite[Proposition I.4.3.(a)]{Schmets}, $f\otimes E$ is complemented (hence closed) in $C_p(X,E)$---a continuous linear projection is given by the mapping $g\mapsto f\otimes g(x)$ since $f(x)=1$ by choice of $f$ and $x$. Note that $f\otimes E$ is  isomorphic to $E$. For every barrel $H\subset C_p(X,E)$ the intersection $H\cap f\otimes E$ is also a barrel in $f\otimes E$. By the assumption, $H\cap f\otimes E$ absorbs $f\otimes B$. This shows that $f\otimes B$ is absorbed by every barrel in $C_p(X,E)$, which means that it is barrel-bounded. Moreover, since $f(x)=1$, we have:
\[
    \sup_{g\in f\otimes B} |\langle\psi_n, g\rangle| =  \sup_{e\in B} |\langle\delta_x\otimes\varphi_n, f\otimes e\rangle| = \sup_{e\in B} |\langle \varphi_n,  f(x)e\rangle| = \sup_{e\in B} |\langle \varphi_n, e\rangle| \not\to 0
  \]
  for $n\to\infty$, which shows that $(\psi_{n})_{n\in\omega}$ is a sequence witnessing the JNP of $C_p(X,E)$.

The proof for the case when $C_p(X)$ has the JNP is analogous.
\end{proof}

We do not know whether the converse to Proposition \ref{JNPCpXE} holds.

\begin{problem}
Let $X$ be a Tychonoff space and $E$ a locally convex space. Assume that the space $C_p(X,E)$ has the JNP. Does it follow that $C_p(X)$ has the JNP or $E$ has the JNP?
\end{problem}

Having in mind Theorem \ref{thm:bg_cpx_jnp} and Proposition \ref{prop:cpxcpy_jnp_seq}, we also ask the following.

\begin{problem}
Can the Josefson--Nissenzweig property of the spaces $C_p(X,E)$ for $X$ Tychonoff and $E$ an arbitrary locally convex space be characterised in terms of some special ``JN-sequences'' in $L_p(X)\otimes E'$, like it is done in the case of $E=C_p(Y)$ for $Y$ Tychonoff?
\end{problem}

\section{Complemented copies of $(c_{0})_{p}$ in $C_p(X,C_p(Y))$\label{sec:c0p_cpxcpy}}

The following theorem of Banakh, K\k{a}kol and \'{S}liwa characterizes in terms of the complemented copies of the space $(c_0)_p$ those infinite Tychonoff spaces $X$ for which the space $C_p(X)$ has the JNP.

\begin{theorem}[{\cite[Theorem 1]{BKS}}]\label{thm:bks}
  Let $X$ be an infinite Tychonoff space. Then, the following conditions are equivalent:
  \begin{enumerate}
  \item $C_p(X)$ has the JNP;
  \item $C_p(X)$ contains a complemented copy of $(c_0)_p$;
  \item $C_p(X)$ admits a continuous linear surjection onto $(c_0)_p$.
  \end{enumerate}
\end{theorem}

\noindent This theorem, together with Theorems \ref{thm:cpxysigma_cpxcpy} and  \ref{thm:cpxysigma_jnp} and the fact that the JNP is preserved under isomorphisms, implies the following criterion for a given space $C_p(X,C_p(Y))$ to contain a complemented copy of $(c_0)_p$. Of course, by Theorem \ref{thm:cpxysigma_cpxcpy}, the following result will remain true if we replace each occurrence of $C_p(X,C_p(Y))$ by $C_p(X\times Y,\sigma)$.

\begin{theorem}\label{thm:cpxcpy_cps_compl}
  Let $X$ and $Y$ be infinite Tychonoff spaces. Then, the following conditions are equivalent:
  \begin{enumerate}
  \item $C_p(X,C_p(Y))$ has the JNP;
  \item $C_p(X,C_p(Y))$ contains a complemented copy of $(c_0)_p$;
  \item $C_p(X,C_p(Y))$ admits a continuous linear surjection onto $(c_0)_p$;
  \item $C_p(X)$ contains a complemented copy of $(c_0)_p$ or $C_p(Y)$ contains a complemented copy of $(c_0)_p$;
  \item $C_p(X)$ admits a continuous linear surjection onto $(c_0)_p$ or $C_p(Y)$ admits a continuous linear surjection onto $(c_0)_p$;
  \item $C_p(X)$ has the JNP or $C_p(Y)$ has the JNP.
  \end{enumerate}
\end{theorem}

\begin{corollary}\label{cor:cpx_c0p_cpxcpx}
  For any infinite Tychonoff space $X$, the space $C_p(X)$ contains a complemented copy of $(c_0)_p$ if and only if $C_p(X,C_p(X))$ contains a complemented copy of $(c_0)_p$.
\end{corollary}

By \cite[Corollary 1.3]{KSZproc}, for any infinite compact space $X$ the space $C_p(X\times X)$ contains a complemented copy of $(c_0)_p$. Since $C_p(\beta\omega)$ does not have any complemented copies of $(c_0)_p$, see~\cite[Corollary~2]{BKS}, Theorem \ref{thm:cpxcpy_cps_compl} yields the following example. 

\begin{example}
  The space $C_p(\beta\omega,C_p(\beta\omega))$ does not have any complemented copy of $(c_0)_p$, even though the space $C_p(\beta\omega\times\beta\omega)$ does have such a copy.
\end{example}

As a corollary to Theorem \ref{thm:cpxcpy_cps_compl} we obtain also the following result providing a large class of counterexamples to Problem \ref{problem:cpxx_compl}.

\begin{corollary}\label{cor:cpx_not_jnp_cpxx_not_compl}
  If $X$ is such an infinite compact space that $C_p(X)$ does not have the JNP, then $C_p(X,C_p(X))$ does not contain any complemented copy of $C_p(X\times X)$.
\end{corollary}
\begin{proof}
  Assume that $C_p(X,C_p(X))$ contains a complemented copy of $C_p(X\times X)$. Since $X$ is compact, by \cite[Corollary 1.3]{KSZproc}, $C_p(X\times X)$ (and hence $C_p(X,C_p(X))$) contains a complemented copy of $(c_0)_p$. By Corollary \ref{cor:cpx_c0p_cpxcpx}, $C_p(X)$ also contains a complemented copy of $(c_0)_p$. By Theorem \ref{thm:bks}, $C_p(X)$ has the JNP, a contradiction.
\end{proof}

\begin{example}
$C_p(\beta\omega\times\beta\omega)$ is not complemented in $C_p(\beta\omega,C_p(\beta\omega))$.
\end{example}

If $X$ is a Tychonoff space and $\kappa$ a cardinal number, then we assume that $C_p(X)^\kappa$ is equipped with the standard product topology, that is, the topology given by basic open sets of the form $\prod_{\xi<\kappa}U_\xi$, where each $U_\xi$ is open in $C_p(X)$ and for all but finitely many $\xi<\kappa$ we have $U_\xi=C_p(X)$. Using Theorem \ref{thm:bks}, it is easy to see that for every infinite Tychonoff space $X$ the space $C_p(X)^2\approx C_p(X\sqcup X)$, where $X\sqcup X$ denotes the topological disjoint union of two copies of $X$, contains a complemented copy of $(c_0)_p$ if and only if $C_p(X)$ contains such a copy. The following proposition generalizes this result to any infinite power of $C_p(X)$.

\begin{proposition}
  Let $X$ be an infinite Tychonoff space and $\kappa$ an infinite cardinal number. Then, the space $C_p(X)^\kappa$ contains a complemented copy of $(c_0)_p$ if and only if $C_p(X)$ contains a complemented copy of $(c_0)_p$.
\end{proposition}
\begin{proof}
  If $C_p(X)$ contains a complemented copy of the space $(c_0)_p$, then trivially $C_p(X)^\kappa$ does it, too, since $C_p(X)$ is complemented in $C_p(X)^\kappa$.

  Conversely, assume that $C_p(X)$ does not contain any complemented copies of $(c_0)_p$. Note that for any infinite discrete space $Y$ the space $C_p(Y)$ does not have the JNP, so if we equip the set $\kappa=\{\xi\colon\ \xi<\kappa\}$ with the discrete topology, then the space $C_p(\kappa)$ does not have the JNP. By Theorem \ref{thm:bks}, it follows that $C_p(\kappa)$ does not contain complemented copies of $(c_0)_p$.  Theorem \ref{thm:cpxcpy_cps_compl} yields that the space $C_p(X,C_p(\kappa))$ does not contain such a copy as well. Since $C_p(\kappa)\approx\mathbb{R}^\kappa$, we get the following equalities:
  \[C_p(X)^\kappa=C_p(X,\mathbb{R})^\kappa\approx C_p(X,\mathbb{R}^\kappa)\approx C_p(X,C_p(\kappa)),\]
  which implies that $C_p(X)^\kappa$ does not contain any complemented copies of the space $(c_0)_p$.
\end{proof}

\section{Complemented copies of $(c_{0})_{p}$ in $C_p(X,E)$\label{sec:c0p_cpxe}}
We start with the following simple observation.

\begin{proposition}\label{copy}
  Let $E$ be a locally convex space and $X$ a Tychonoff space containing an infinite  compact subset. Then $C_{p}(X,E)$ contains a closed copy of the space $(c_0)_p$.
\end{proposition}
\begin{proof}
It is known that for such a space $X$, the space $C_p(X)$ contains a closed copy of the space $(c_0)_p$ (see the proof of \cite[Theorem 5]{kakol-sliwa} for $\mathbb{K}=\mathbb{R}$). The thesis follows now from the fact that $C_p(X)$ is complemented in $C_p(X,E)$ (see~\cite[Theorem I.4.4]{Schmets}).
\end{proof}

Let us now prove the main result of this section. The proof is inspired by the arguments presented in the proof of Theorem 1 in Doma\'nski and Drewnowski \cite{DD}.

\begin{theorem}\label{the1}
  Let $E$ be a barrelled locally convex space with the Josefson--Nissenzweig property and let $X$ be a Tychonoff  space containing an infinite compact set. Then $C_{p}(X, E)$ contains a complemented copy of $(c_0)_{p}$.
\end{theorem}

\begin{proof}
  Since $E$ has the JNP, there is a weak$^*$ null sequence~$(x_n^*)_{n\in\omega}$ in $E'$ which is not convergent in the $\beta^*(E',E)$-topology. Taking into account that $E$ is barrelled this means that the sequence is not strongly convergent and hence we may choose a bounded sequence $(x_n)_{n\in\omega}$ in $E$ with $\langle x_n^*, x_n\rangle = 1$ for every $n\in\omega$. 
  Let $K\subset X$ be an infinite compact set. Using the Tychonoff property, we may choose a sequence of continuous functions $\varphi_n\colon X \to [0,1]$ with disjoint supports and $\varphi_n(t_n)=1$ for some $t_n\in K$.

  We set
  \[
    J\colon (c_0)_{p} \to C_p(X,E), \qquad (a_n)_{n\in\omega}\ \substack{J\\\longmapsto\\~} \sum_{n\in\omega} a_n\varphi_n(t) x_n
  \]
  and
  \[
    P\colon C_p(X,E) \to (c_0)_{p}, \qquad f\substack{P\\\longmapsto\\~} \big(\langle x_n^*, f(t_n)\rangle\big)_{n\in\omega}.
  \]
  The mapping $J$ is well-defined, since the sequence $(a_n x_n)_{n\in\omega}$ is bounded and the functions $\varphi_n$ have disjoint supports---this implies that the series converges uniformly on $X$, and hence it defines a continuous function from $X$ into $E$. Note that the property of the supports of $\varphi_n$'s implies that this function only takes values of the form $a_n\varphi_n(t)x_n$. Let us see why $P$ is well-defined, too. Since $f$ is continuous the set $f(K)\subset E$ is compact. Since on equicontinuous subsets of $E'$ the weak$^*$ topology coincides with the topology of uniform convergence on precompact subsets of $E$ (see e.g.~\cite[Proposition~3.9.8]{Horvath}) and the sequence $(x_n^*)_{n\in\omega}$, being a weak$^*$ null sequence, is equicontinuous, we may conclude that $\langle x_n^*, e\rangle \to 0$ uniformly for $e\in f(K)$.

  We now check that these mappings are continuous.  Let $p$ be a continuous semi-norm on $E$ and $C>0$ such that $p(x_n)\leq C$ for all $n$ (such $C$ exists, since $(x_n)_{n\in\omega}$ is a bounded sequence).  For each $t\in X$,  since the functions $\varphi_n$ have disjoint supports, there is at most one $m_t$ with $\varphi_{m_t}(t)\neq 0$. We pick this $m_t$ if it exists and set $m_t=1$ otherwise (of course, any other $m_t$ works as well in this case), and obtain that for every $(a_n)_{n\in\omega}\in(c_0)_p$ and $t\in X$ we have:
  \[
    p\Big(J((a_n)_{n\in\omega})(t)\Big) = p\Big(\sum_{n\in\omega}a_n\varphi_n(t)x_n\Big) \leq C |a_{m_t}| \leq C \max_{n\leq N} |a_n|
  \]
  for all $N\geq m_t$. It follows that $J$ is continuous. Now, let $f\in C_p(X,E)$ be given and $N\in\omega$ be arbitrary. Observe that since the functionals $x_n^*$ are continuous there is a continuous semi-norm $p$ on $E$ and $\tilde{C}>0$ with 
  \[
    \max_{n\leq N} |(P(f))_n| = \max_{n\leq N} |\langle x_n^*, f(t_n)\rangle| \leq \tilde{C} \max_{n\leq N} p(f(t_n)),
  \]
  which implies the continuity of $P$.

  Since
  \[
    P(J((a_n)_{n\in\omega})) = \Big(\Big\langle x_n^*, \sum_{m=1}^{\infty} a_m\varphi_m(t_n)x_m\Big\rangle \Big)_{n\in\omega} = (a_n\langle x_n^*,x_n\rangle)_{n\in\omega} = (a_n)_{n\in\omega},
  \]
  we have that $J$ is an isomorphic embedding and $P$ is a surjection. Since $JPJP=JP$, the mapping $JP$ is a projection onto the image of $J$.
\end{proof}

\begin{remark}
Note the following comments concerning Theorem~\ref{the1}.

(1) If $X$ is infinite discrete and $E$ is a Banach space, then $C_p(X,E)\approx E^X$, which does not contain any closed copies of $(c_0)_p$, as the latter is not complete.

(2) The existence of an infinite compact subspace in $X$ is not necessary for Theorem~\ref{the1}. Indeed,  if $X$ is a pseudocompact space with all compact subsets being finite and $K$ is an infinite compact space, then  $C_p(X,C(K))\approx C_k(X,C_k(K))\approx C_k(X\times K)$, 
which, by~\cite[Proposition~2.10]{KMSZ}, contains a complemented copy of $(c_0)_p$.

(3) If $C_p(Y)$ has the JNP for an infinite Tychonoff space $Y$, then $C_p(Y)$  is not barrelled (by the Buchwalter--Schmets theorem \cite{BS} and \cite[Proposition 4.1]{KMSZ}), so it does not satisfy the assumptions of Theorem \ref{the1}, even though the space $C_p(X,C_p(Y))$ still has the JNP for every space $X$ (by Theorem \ref{thm:cpxcpy_cps_compl}).
\end{remark}

As a corollary we immediately obtain the following $C_p$-variant of  Doma\'nski--Drewnowski's Theorem \ref{dom-drew}. 

\begin{corollary}\label{variant_dd}
  Let $X$ be a Tychonoff  space containing an infinite compact set and $E$ a Fr\'{e}chet space which is not a Montel space. Then $C_{p}(X, E)$ contains a complemented copy of $(c_0)_p$.
\end{corollary}

\begin{proof}
  Since by~\cite[Theorem~2.3]{BG} and \cite{BLV} a Fr\'{e}chet space has the JNP if and only if it is not a Montel space, the claim is a special case of Theorem~\ref{the1}.
\end{proof}

Using the Closed Graph Theorem one can easily derive a special case of Theorem \ref{dom-drew} for Fr\'echet spaces $C_k(X,E)$. Thus, if $X$ is compact and $E$ infinite-dimensional Banach, we obtain Cembranos' and Freniche's Theorem \ref{Cembranos}.

\begin{corollary}\label{first}
Let $X$ be a Tychonoff space containing an infinite compact subspace and let $E$ be a Fr\'echet space which is not Montel. If  $C_{k}(X,E)$  is  an infinite-dimensional  Fr\'echet space, then $C_k(X,E)$ contains a complemented copy of the Banach space $c_0$.
\end{corollary}

\begin{proof}
Let $\tau_p$ and $\tau_{k}$ denote the pointwise topology of the space $C_p(X,E)$ and the compact-open  topology of $C_k(X,E)$, respectively. By Theorem \ref{the1}, $C_{p}(X,E)$ contains a (closed) complemented subspace $G$ which is isomorphic to $(c_{0})_{p}$, i.e. $(G,\tau_{p}\restriction G)\approx (c_{0})_{p}$. 
Let $\xi$ be the Banach space topology on $G$ such that $\tau_{p}\restriction G\subseteq \xi$ and $(G,\xi)\approx c_{0}$. The identity map
  $I\colon(G,\tau_{p}\restriction G)\to (G,\tau_{k}\restriction G)$ has closed graph, so also $I\colon(G,\xi)\to(G,\tau_{k}\restriction G)$ has closed graph. The  Closed Graph Theorem (\cite[Theorem 4.1.10]{bonet})
  assures that both $I$ and $I^{-1}$ are continuous, so $(G,\xi)=(G,\tau_{k}\restriction G)$, which completes the proof.
\end{proof}

Theorem \ref{the1} yields also the following dichotomy.

\begin{corollary} \label{mon}
  Let $X$ and $Y$ be  infinite Tychonoff spaces. Assume that $Y$ is compact. Then $C_{p}(X,C(Y))$ contains a complemented copy of $\mathbb{R}^{\omega}$ or a complemented copy of $(c_{0})_{p}$.
\end{corollary}
\begin{proof}
We need to consider the following three cases:

  Case 1. \emph{$X$ is not pseudocompact.}  Then $C_p(X, C(Y))$ contains a complemented copy of $\mathbb{R}^{\omega}$ (since  $C_{p}(X)$ contains such a copy and is complemented in $C_{p}(X,C(Y))$).

  Case 2. \emph{$X$ is pseudocompact and $X$ contains an infinite compact subset}. We apply Theorem \ref{the1} to get a complemented copy of $(c_0)_p$.

  Case 3. \emph{$X$ is pseudocompact and every compact set in $X$ is finite}. Since $X\times Y$ is pseudocompact, we apply \cite[Proposition 2.10]{KMSZ} to deduce that $C_p(X,C(Y))\approx C_{k}(X\times Y)$ contains a complemented copy of $(c_{0})_p$.  
\end{proof}

Recall that a Tychonoff space $X$ is \textit{$\sigma$-compact} if $X$ is covered by an increasing sequence of compact subsets.
We provide the following extension of \cite[Proposition 2.10]{KMSZ}.

\begin{corollary}\label{second_cor}
  Let $X$ and $Y$ be infinite locally compact $\sigma$-compact spaces. Then $C_{k}(X\times Y)$ contains a complemented copy of $\mathbb{R}^{\omega}$ or a complemented copy of the  Banach space $c_{0}$.
\end{corollary}
\begin{proof}
By \cite[Corollary 2.5.7]{McCoy}, the spaces $C_{k}(X\times Y)$ and $C_{k}(X,C_{k}(Y))$ are isomorphic.  Assume first that $Y$ is non-discrete. Since $C_{k}(Y)$ is not Montel (see \cite[Theorem 11.7.7]{Jarchow}) and it is  a Fr\'echet space \cite[10.1.25]{bonet}, we apply Corollary \ref{first} to obtain a complemented copy of the Banach space $c_0$ (note that the space $C_k(X,C_k(Y))$ is Fr\'echet by \cite[Theorem IV.1.3.(b)]{Schmets}). On the other hand, if $Y$ is discrete, then $Y$ is countable, and hence the space $\mathbb{R}^\omega=C_{k}(Y)$ is complemented in $C_{k}(X, C_{k}(Y))\approx C_{k}(X\times Y)$.
\end{proof}

\begin{remark}
If $X$ is a locally compact and $\sigma$-compact space without  infinite compact subsets, then $X$ must necessarily be discrete and countable, and hence homeomorphic to $\omega$. It follows that for such $X$ and any locally convex space $E$, the space $C_p(X,E)=C_k(X,E)$ contains a complemented copy of $\mathbb{R}^\omega$ (cf. Case 1 in the proof of Corollary \ref{mon}). 
\end{remark}

Since $(c_0)_p\approx C_p(S)$ is linearly homeomorphic to  $C_p(S)\times\mathbb{R}$, we obtain the following application of Corollary~\ref{first}.
\begin{corollary} \label{second}
  If $X$ and $Y$ are as in Corollary \ref{mon}, then the space  $C_{p}(X,C(Y))$ is isomorphic to the space $C_{p}(X,C(Y))\times\mathbb{R}$.
\end{corollary}
The above result may be thought of as a (partial) positive answer to a problem of Arkhangel'ski for spaces of the form $C_p(X,E)$. Recall that Arkhangel'ski (see \cite{arkh4}) asked if it is true that for every infinite (compact) space $K$ the space  $C_{p}(K)$ is linearly homeomorphic to  $C_{p}(K)\times\mathbb{R}$. For a wide class of spaces the answer to this question is affirmative, e.g. if the space $X$ contains a non-trivial convergent sequence, or $X$ is not pseudocompact (see \cite[Section 4]{arkh4}), yet, in general, the answer is negative even for compact spaces $K$---namely, Marciszewski \cite{marci} showed that
there exists an infinite compact space $K$ such that $C_{p}(K)$ cannot be mapped onto  $C_{p}(K)\times\mathbb{R}$ by a continuous linear surjection.

We finish this section with the following natural  question motivated by  Theorems \ref{thm:bks} and \ref{thm:cpxcpy_cps_compl}.

\begin{problem}
Let $X$ be a Tychonoff space (containing an infinite compact subspace) and $E$ a locally convex space. Assume that $C_p(X,E)$ has the JNP. Does it follow that $C_p(X,E)$ contains a complemented copy of $(c_0)_p$?
\end{problem}

\section{Complemented copies of $(c_{0})_{p}$ in $C_p(X,E_{w})$\label{sec:c0p_cpxew}}

In order to provide a ``weak'' version of Theorem \ref{the1}, we need the following auxiliary result proved by Doma\'nski and Drewnowski (\cite[Lemma 1]{DD}). The lemma plays a similar role in proving Theorem \ref{the1-1} as \cite[Theorem 2]{BLV} does in the proof of Theorem \ref{the1}. As usual, by $\ell_1$ we mean the standard Banach space of all summable real-valued sequences.

\begin{lemma}[{\cite[Lemma 1]{DD}}]\label{lemma:dd1}
  Let $E$ be a locally convex space containing a copy of the Banach space $\ell_{1}$. Then there exists an equicontinuous sequence $(u_{n})_{n}$ in $E'$ which converges uniformly on each weakly compact set in $E_{w}$  but on some bounded set $B$ in $E$ one has
  $\sup_{x\in B}|u_{n}(x)|=1$ for each $n\in\omega$.
\end{lemma}

Following the main idea of the proof of Theorem \ref{the1}, Lemma \ref{lemma:dd1} can be used to obtain the next result.
\begin{theorem}\label{the1-1}
  Let $X$ be a Tychonoff space which contains an infinite compact subset. Let $E$ be a locally convex space which contains an isomorphic copy of the Banach space $\ell_{1}$. Then $C_{p}(X,E_{w})$ contains a complemented copy of $(c_0)_{p}$.
\end{theorem}

K\k{a}kol and Moll \cite[Theorem~3]{Kakol-Moll} showed that for a Tychonoff space $Y$ the space $C_{k}(Y)$ does not contain a copy of $\ell_{1}$ if and only if every compact subset of~$Y$ is scattered. Hence, we have the following
\begin{corollary}\label{sca}
  Let $X$ be a Tychonoff space which contains an infinite compact subset and let $Y$ be an infinite Tychonoff space containing a non-scattered compact subset. Then $C_{p}(X,C_{k}(Y)_{w}))$ contains a complemented copy of $(c_0)_{p}$.
\end{corollary}

Recall that a Banach space $E$ is \textit{a Grothendieck space} if every weakly$^*$ convergent sequence in the dual space $E'$ is weakly convergent. If $X$ is a compact space such that $C(X)$ is a Grothendieck space (e.g. if $X=\beta\omega$ or $X=\omega^*$), then $X$ is non-scattered (as otherwise it would contain a non-trivial convergent sequence and hence $C(X)$ would not be Grothendieck) and $C_{p}(X)$ does not contain a complemented copy of $(c_0)_{p}$ (see \cite[Section 4]{KSZfm}). It follows that both the spaces $C_p(X,C_k(X)_w)=C_p(X,C(X)_w)$ and $C_p(X,C(X))$ contain a complemented copy of $(c_0)_p$ (Corollary \ref{sca} and Theorem \ref{the1}, respectively), even though the space $C_p(X,C_p(X))$ does not contain such a copy (Corollary \ref{cor:cpx_c0p_cpxcpx}).

The next corollary implies that there is no continuous linear surjection from the space $C_p(\beta\omega,C_p(\beta\omega))$ onto $C_p(\beta\omega,C(\beta\omega)_w)$ (cf. Problem \ref{prob:cpxy_cpxcyw}).
\begin{corollary}\label{scattered}
  Let $X$ be a non-scattered compact space such that $C_{p}(X)$ does not contain a complemented copy of $(c_0)_{p}$. Then there is no continuous linear surjection from $C_{p}(X,C_{p}(X))$ onto $C_{p}(X,C(X)_{w})$.
\end{corollary}
\begin{proof}
For the sake of contradiction assume that there exists a continuous linear surjection from $C_{p}(X,C_{p}(X))$ onto $C_{p}(X,C(X)_{w})$. Since $X$ is non-scattered, by Corollary \ref{sca}, $C_{p}(X, C(X)_{w})$ contains a complemented copy of~$(c_0)_{p}$, hence  there exists a continuous linear surjection from $C_{p}(X, C(X)_{w})$ onto $(c_0)_{p}$.   It follows that $C_{p}(X,C_{p}(X))$ can be mapped onto $(c_0)_{p}$ by a continuous linear map. We apply Theorem~\ref{thm:cpxcpy_cps_compl} to derive that $C_{p}(X)$ contains a complemented copy of $(c_0)_{p}$, a contradiction. \end{proof}

The last corollary yields a variant of Krupski--Marciszewski's result (\cite[Corollary 3.2]{KM}), already mentioned in Section \ref{sec:cpxy_vs_cpxcpy}.
\begin{corollary}
  Let $X$ be an infinite compact space. Then the spaces $C_{p}(X)$ and $C(X)_{w}$ are not isomorphic.
\end{corollary}
\begin{proof}
Assume that an isomorphism $T\colon C_p(X)\to C(X)_w$ exists. Consider two cases:

  Case 1. \emph{$X$ is not scattered}. Define a linear map $S\colon C_{p}(X, C_{p}(X))\to C_{p}(X, C(X)_{w})$ 
  by the formula $S(f)=T\circ f$ for each continuous $f\colon X\to C_{p}(X)$. By \cite[Theorem 2.2.4]{McCoy}
  the map $S$ is continuous. Note that $S$ is a surjection. Indeed, choose arbitrary continuous  $h\colon X\to C(X)_{w}$ and define the continuous function $f(x)=T^{-1}(h(x))$ for each $x\in X$---clearly, $S(f)(x)=T(f(x))=h(x)$ for every $x\in X$.

  By  Corollary~\ref{scattered}, the space $C_{p}(X)$ contains a complemented copy of $(c_0)_{p}$. Since $T$ is an isomorphism, also the space $C(X)_{w}$ contains a complemented  copy $F$ of $(c_0)_{p}$. As  both the weak topology $\tau_w$ of $C(X)$ and the normed topology $\tau_\infty$ of  $C(X)$ have the same bounded sets, we have $\tau_w\restriction F=\tau_\infty\restriction F$, a contradiction (since  $C(X)_{w}$ does not contain an infinite-dimensional normed subspace).

  Case 2. \emph{$X$ is scattered}. By~\cite[Theorem III.1.2]{Arch} the space $C_{p}(X)$ is Fr\'echet--Urysohn. Hence, the space $C(X)_{w}$ is also Fr\'echet--Urysohn. By \cite[Lemma 14.6]{kak} the weak and the Banach topologies of $C(X)$ coincide  which implies that $X$ is finite, a contradiction.
\end{proof}

\noindent\textbf{Acknowledgements.} The first and the third named authors are supported by the Austrian Science Fund (FWF):~I 4570-N. The second named author is supported by the GA\v{C}R project 20-22230L  RVO: 67985840.

\end{document}